\documentclass[english]{article}

\usepackage{babel}          
\usepackage[utf8]{inputenc}  
\usepackage[T1]{fontenc}   
\usepackage[ec]{aeguill}      
\usepackage{amsmath, amsthm} 
\usepackage{amsfonts,amssymb}
\usepackage{float}           
\usepackage[pdftex]{graphicx}
\usepackage{epstopdf}
\usepackage{fancyhdr}     
\usepackage{url}            
\usepackage{hyperref}      

\numberwithin{equation}{section}

\newcommand{\TT}{\mathbb{T}}
\newcommand{\DD}{\mathbb{D}}
\newcommand{\RR}{\mathbb{R}}

\newcommand{\NN}{\mathbb{N}}

\newcommand{\dd}{{\rm d}}
\renewcommand{\u}{{\bf u}}
\renewcommand{\v}{{\bf v}}

\renewcommand{\k}{{\bf k}}

\renewcommand{\r}{{\rm reg}}
\newcommand{\s}{{\rm sing}}

\newcommand{\A}{{\sf A}}
\newcommand{\J}{{\sf J}^\mu}
\renewcommand{\L}{{\sf L}}

\renewcommand{\O}{\mathcal{O}}
\renewcommand{\S}{{\sf S}}

\DeclareMathOperator{\Id}{{\sf Id}}
\DeclareMathOperator{\curl}{curl}

\newcommand{\ie}{{\em i.e.}~}
\newcommand{\eg}{{\em e.g.}~}
\newcommand{\id}[1]{\left\vert_{_{#1}}\right.}

\usepackage{mathtools}
\DeclarePairedDelimiter\norm{\big\lvert}{\big\rvert}
\DeclarePairedDelimiter\Norm{\big\lVert}{\big\rVert}
\newcommand{\eqdef}{\stackrel{\rm def}{=}}

\newtheorem{Theorem}{Theorem}[section]
\newtheorem{Proposition}[Theorem]{Proposition}

\newtheorem{Lemma}[Theorem]{Lemma}

\newtheorem{Remark}[Theorem]{Remark}

\title{Rigorous justification of the Favrie-Gavrilyuk approximation to the Serre-Green-Naghdi model}
\author{Vincent Duchêne%
\thanks{IRMAR - UMR6625, CNRS and Univ. Rennes 1. \href{mailto:vincent.duchene@univ-rennes1.fr}{vincent.duchene@univ-rennes1.fr}}
}
\date{\today}
\begin{document}
\thispagestyle{empty}
\maketitle

\begin{abstract}
The (Serre-)Green-Naghdi system is a non-hydrostatic model for the propagation of surface gravity waves in the shallow-water regime. Recently, Favrie and Gavrilyuk proposed in~\cite{FavrieGavrilyuk17} an efficient way of numerically computing approximate solutions to the Green-Naghdi system. The approximate solutions are obtained through solutions of an augmented quasilinear system of balance laws, depending on a parameter. In this work, we provide quantitative estimates showing that any regular solution of the Green-Naghdi system is the limit of solutions to the Favrie-Gavrilyuk system as the parameter goes to infinity, provided the initial data of the additional unknowns is well-chosen. The problem is therefore a singular limit related to low Mach number limits with additional difficulties stemming from the fact that both order-zero and order-one singular components are involved.
\end{abstract}

\section{Introduction}

\subsection{Motivation}

The (Serre-)Green-Naghdi system arises as a model for the propagation of weakly dispersive surface gravity waves. It has been derived many times in the litterature, and in particular in~\cite{Serre53,SuGardner69,GreenNaghdi76,MilesSalmon85,Seabra-SantosRenouardTemperville87}. More recently~\cite{Lannes}, it has been rigorously justified as an asymptotic model for the so-called water-waves system in the shallow-water regime. It can be seen as a second-order model refining the Saint-Venant system so as to take into account dispersive effects, and as such has received a fair amount of attention. Let us write down one of the many formulations of the Green-Naghdi system. Denoting $h(t,X)$ the depth of the water and $\u(t,X)$ the layer-averaged horizontal velocity at time $t\in\RR$ and horizontal position $X\in\DD^d$ (where $\DD=\RR$ or $\DD=\TT$ and $d\in\{1,2\}$), the Green-Naghdi system reads
\begin{equation}\label{GN}
\left\{\begin{array}{l}
\partial_t h+\nabla\cdot(h\u)=0,\\ \\
\partial_t \u+g\nabla h+(\u\cdot\nabla)\u+ \frac1{3h}\nabla\big(h^2\ddot{h}\big)={\bf 0},
\end{array}\right.
\end{equation}
where  $g$ is the gravitational acceleration and denoting $\dot h=\partial_t h+\u\cdot \nabla h$ and $\ddot{h}=\partial_t \dot h+\u\cdot \nabla\dot h$.

A difficulty arises when one tries to ---numerically or analytically--- solve the initial-value problem associated with~\eqref{GN} as, after using the equation of mass conservation to rewrite $\ddot h$, it is found necessary to invert the elliptic operator
\[\mathfrak T[h]:\v  \mapsto \v-\frac1{3h}\nabla\big(h^3\nabla\cdot\v).\]
This is only a technical difficulty in the proof of the local well-posedness of the Cauchy problem \cite{Li06,Alvarez-SamaniegoLannes08a,FujiwaraIguchi15,DucheneIsrawi}, but remains a severe issue for practical numerical simulations, as the cost of inverting this operator at each time step can be prohibitive, especially in dimension $d=2$. We refer to~\cite{LeGavrilyukHank10,MitsotakisSynolakisMcGuinness17} and references therein for several numerical schemes adapted to the Green-Naghdi system. The aforementioned issue is addressed in~\cite{LannesMarche15,DuranMarche15}, where the authors introduce a new class of models which enjoy the same precision as the original Green-Naghdi system ---as an asymptotic model for the water-waves systems--- but for which the elliptic operator playing the role of $\mathfrak T[h]$ is independent of time.
A different direction of investigation is proposed in the recent paper by Favrie and Gavrilyuk~\cite{FavrieGavrilyuk17}. By modifying the lagrangian associated with the variational formulation of the Green-Naghdi system, the authors derive a system of balance laws depending on a an augmented set of unknowns and on a free parameter:
\begin{equation}\label{FG}
\left\{\begin{array}{l}
\partial_t h+\nabla\cdot(h\u)=0,\\
\partial_t \u+g\nabla h+(\u\cdot\nabla)\u-\frac\lambda{3h} \nabla\big( \frac\eta{h}(\eta-h)\big)={\bf 0},\\
\partial_t\eta+\u\cdot\nabla\eta=w,\\
\partial_tw+\u\cdot\nabla w=-\frac\lambda{h^2}\big(\eta-h\big).
\end{array}\right.
\end{equation}
 The claim is that in the limit $\lambda\to \infty$, solutions to~\eqref{FG} approach solutions to~\eqref{GN}. Indeed we expect, using the fourth and third equations of~\eqref{FG}:
\[\eta= h+\O(\lambda^{-1}) \quad \text{ and } \quad  \lambda \big(\eta-h\big)= -h^2 \ddot{\eta}= -h^2 \ddot{h}+\O(\lambda^{-1}),\]
and we recover~\eqref{GN} when plugging the truncated approximations in the second equation of~\eqref{FG}.

{\em The aim of this work is to produce quantitative estimates which allow to rigorously prove that the Favrie-Gavrilyuk system~\eqref{FG} generates arbitrarily close solutions to the Green-Naghdi system~\eqref{GN} on the relevant timescale.}

Among other things, our work gives insights to how large $\lambda$ should be set and how initial data for $\eta$ and $w$ should be constructed in order for the corresponding solution to~\eqref{FG} to be a valid approximation to the solution of the Green-Naghdi system~\eqref{GN}; and hence to surface gravity waves in the shallow-water regime.

\subsection{Main results}

The main tool for our results will be energy estimates, which need to be assessed uniformly with respect to the parameter $\lambda$. In order to provide useful results, we also need to provide estimates which are uniform with respect to the other parameters of the system, and in particular the (small) shallowness parameter, which measures the precision of the Green-Naghdi system.

Hence for proper comparison, we start by non-dimensionalizing the systems~\eqref{GN} and~\eqref{FG}. A natural choice of scaling in the shallow-water regime
\footnote{We use the scaled variables
\[{\bf x} \leftarrow {\bf x}/L \quad ; \quad t \leftarrow t\times \sqrt{gh_0}/L, \]
and scaled unknowns
\[\u \leftarrow \u/\sqrt{gh_0} \quad ; \quad h\leftarrow h/h_0 \quad ; \quad b\leftarrow b/h_0 \quad ; \quad \zeta \leftarrow \zeta/h_0.\]

The choice is less obvious for the augmented unknowns which have no direct physical interpretation. In view of their dimension and the expected behaviour as $\lambda\to\infty$, we set
\[ w \leftarrow w/\sqrt{gh_0}\times(L/h_0) \quad ; \quad \eta\leftarrow \eta/h_0 .\]
Thus we scale $w$ differently from $\u$, because the former represents typically a vertical velocity while the latter is the layer-averaged horizontal velocity.
As for $\lambda$, we set
\[\lambda\leftarrow \lambda/(gh_0)\times (L/h_0)^2,\]
Here again this choice reflects the fact that $\lambda$ compares with a vertical acceleration, times a vertical length.
}
 yields respectively
\begin{equation}\label{GN-adim}
\left\{\begin{array}{l}
\partial_t \zeta+\nabla\cdot(h\u)=0,\\ \\
\partial_t \u+\nabla \zeta+(\u\cdot\nabla)\u+ \frac\mu{3h}\nabla\big(h^2\ddot{h}\big)={\bf 0},
\end{array}\right.
\end{equation}
and
\begin{equation}\label{FG-adim}
\left\{\begin{array}{l}
\partial_t \zeta+\nabla\cdot(h\u)=0,\\
\partial_t \u+\nabla \zeta+(\u\cdot\nabla)\u-\frac{\lambda\mu}{3h} \nabla\big( \frac\eta{h}(\eta-h)\big)={\bf 0},\\
\partial_t\eta+\u\cdot\nabla\eta=w,\\
\partial_tw+\u\cdot\nabla w=-\frac{\lambda}{h^2}\big(\eta - h\big).
\end{array}\right.
\end{equation}
Above, $\zeta$ is the dimensionless surface deformation and we will always denote $h=1+\zeta$. The shallowness parameter $\mu$ is the square of the ratio of the typical depth of the layer to the typical horizontal wavelength of the wave, and is assumed to be small in the shallow-water regime: roughly speaking, regular solutions to the Green-Naghdi system~\eqref{GN-adim} approximate corresponding solutions to the water-waves system up to an error of size $\O(\mu^2 t)$ on the ``quasilinear'' time-scale \ie up to a maximal time inversely proportional to the size of the initial data; see~\cite{Lannes}. We aim at proving that solutions to system~\eqref{FG-adim} when $\lambda$ is large and initial data for $(\eta,w)$ are well-prepared approach solutions to the Green-Naghdi system~\eqref{GN-adim}, uniformly with respect to the parameter $\mu$ and on the quasilinear time-scale.

Because system~\eqref{FG-adim2} is a symmetrizable hyperbolic quasilinear system (as it is checked in Section~\ref{S.Hyp} below), the well-posedness of the corresponding initial-value problem is provided by standard theory; see \eg~\cite{Benzoni-GavageSerre07}.
\begin{Theorem}\label{T.WP0}
Let $s\in\RR$ with $s>1+d/2$. Then for any $\lambda,\mu\in(0,\infty)$ and any $U_0=(\zeta_0,\u_0,\eta_0-1,w_0)\in H^s(\DD^d)^{d+3}$ satisfying $h_0=1+\zeta_0\geq h_\star>0$, there exists a maximal time $T^\star>0$ and 
$U=(\zeta,\u,\eta-1,w)\in C^1([0,T^\star)\times\DD^d)^{d+3}$ unique maximal strong solution to~\eqref{FG-adim2} with $U\id{t=0}=U_0$. Moreover, one has $U\in \bigcap_{j=0}^s C^j([0,T^\star);H^{s-j}(\DD^d)^{d+3})$, and
$T^\star=\infty$ or $\lim_{t \nearrow  T^\star}\norm{U}_{W^{1,\infty}}(t)=\infty$.
\end{Theorem}
Solutions to the Favrie-Gavrilyuk system are valuable approximations to the Green-Naghdi system (in the sense of consistency) as long as several space an time derivatives of the solutions are uniformly bounded, as stated below.
\begin{Theorem}\label{T.consistency}
Let $s>d/2$ with $s\geq 2$,  $\lambda,\mu\in(0,\infty)$ and $U=(\zeta,\u,\eta-1,w)\in C^2([0,T^\star);H^{s-2}(\DD^d)^{d+3})\cap C^1([0,T^\star);H^{s-1}(\DD^d)^{d+3})\cap C^0([0,T^\star);H^{s}(\DD^d)^{d+3})$ a solution to~\eqref{FG-adim} satisfying $h=1+\zeta\geq h_\star>0$. Then for any $t\in[0,T^\star)$, $(\zeta,\u)$ satisfies
\begin{equation}\label{GN-adim-r}
\left\{\begin{array}{l}
\partial_t \zeta+\nabla\cdot(h\u)=0,\\ \\
\partial_t \u+\nabla \zeta+(\u\cdot\nabla)\u+ \frac\mu{3h}\nabla\big(h^2\ddot{h}\big)=\frac{\mu}{3h}\nabla r,
\end{array}\right.
\end{equation}
with $ r=-h(\eta-h)\ddot{\eta}- h^2(\ddot{\eta}-\ddot{h})\in C^0([0,T^\star);H^{s-2}(\DD^d))$ and therefore
\[ \forall t\in[0,T^\star), \qquad \norm{r}_{H^{s-2}}(t)\leq  C\sum_{j=0}^2\norm{\partial_t^j(\eta-h)}_{H^{s-j}}(t).\] 
with $C=C(h_\star^{-1},\norm{U}_{H^s},\norm{\partial_t U}_{H^{s-1}},\norm{\partial_t^2 U}_{H^{s-2}})$. Moreover, if the right-hand-side is finite, then
\[ \forall t\in[0,T^\star), \qquad\norm{r}_{H^{s-2}}(t)\leq  \lambda^{-1} C' \sum_{j=0}^3 \norm{\partial_t^jw}_{H^{s+1-j}}(t)\]
with $C'=C(h_\star^{-1},\norm{U}_{H^s},\norm{\partial_t U}_{H^{s-1}},\norm{\partial_t^2 U}_{H^{s-2}})$. 
\end{Theorem}
\begin{proof}
The formula for $r$ comes from straightforward manipulations on~\eqref{FG-adim}, and the first estimate follows from Lemma~\ref{L.products} after expanding $\ddot{\eta}$, $\ddot{h}$. The second estimate is obtained applying Lemma~\ref{L.products} to the last equation of~\eqref{FG-adim}.
\end{proof}
Of course, there is no reason to hope {\em a priori} that the maximal solutions to~\eqref{FG-adim} with initial data in a given ball of $H^s(\DD^d)^{d+3}$ --- or continuously embedded normed spaces --- satisfy the estimates of Theorem~\ref{T.consistency} on a relevant timescale uniformly with respect to the parameters $\lambda$ (large) and $\mu$ (small). The main result of this work is to prove that it is possible to {\em prepare the initial data} so that such property holds.

All our results from now on are restricted to the following set of parameters
\begin{equation}\label{not-Snu}
\mathcal{S}_{\nu}=\{ (\lambda,\mu)\in (0,\infty)^2,\ \lambda^{-1}+\mu+(\lambda\mu)^{-1}\leq \nu\}.
\end{equation}
The first two restrictions in $\mathcal{S}_\nu$ are harmless in our framework, but the second one already hints at a possibly non-uniform behaviour with respect to small values of $\mu$.
Let us further prepare the Favrie-Gavrilyuk system through a change of variables which allows to balance the singular terms in~\eqref{FG-adim}. Introducing
\begin{equation} \label{change}\iota\eqdef (\mu\lambda)^{1/2}(\eta-h) \quad ; \quad \kappa \eqdef \mu^{1/2} h^{-1}w ,\end{equation}
we see that~\eqref{FG-adim} is equivalent to
\begin{equation}\label{FG-adim2}
\left\{\begin{array}{l}
\partial_t \zeta+\nabla\cdot(h\u)=0,\\
\partial_t \u+(\u\cdot\nabla)\u+\nabla\zeta-\frac{1}{3h} \nabla\Big( (\mu\lambda)^{1/2}\iota+ \frac{\iota^2}{h}\Big)={\bf 0},\\
\partial_t\iota+\u\cdot\nabla\iota=\lambda^{1/2}(h\kappa+ \mu^{1/2}h\nabla\cdot\u),\\
\partial_t(h\kappa)+\u\cdot\nabla (h\kappa)=-\lambda^{1/2}h^{-2}\iota.
\end{array}\right.
\end{equation}
The following result shows that one can control solutions to~\eqref{FG-adim} on a time interval uniform with respect to $\lambda$ (sufficiently large) and $\mu$ provided that the initial data is well-prepared.
\begin{Theorem} \label{T.WP-mr}
Let $m,s\in\NN$ with $s>1+d/2$, $1\leq m\leq s$, and $h_\star,M_0,\nu>0$. Set also $\delta_\star\in(0,1)$ if $m=s$. There exist $\nu_\star,\tau,C_0>0$ such that for any $(\lambda,\mu)\in\mathcal{S}_\nu$ satisfying $\lambda\mu\geq \nu_\star$, for any $\tilde\lambda\in[1,\lambda\mu]$ and for any $V\in C^0([0,T^\star);H^s(\DD^d))$ maximal strong solution to~\eqref{FG-adim2} such that one has $h\id{t=0}=1+\zeta\id{t=0}\geq h_\star$ and
\[\sum_{j=0}^{m} \norm{\partial_t^j V}_{H^{s-j}}(0)+\sum_{j=m+1}^{s} \tilde\lambda^{\frac{m-j}2}\norm{\partial_t^j V }_{H^{s-j}} (0)\leq M_0,\]
 one has $T^\star > ( M_0\tau )^{-1}$ and for any $t\in[0, (M_0 \tau)^{-1}]$,
\[\sum_{j=0}^{m} \norm{\partial_t^j V}_{H^{s-j}}(t)+\sum_{j=m+1}^{s} \tilde\lambda^{\frac{m-j}2}\norm{\partial_t^j V }_{H^{s-j}} (t) \leq C_0 M_0  .\]
If $m=s$, we can withdraw the condition $\lambda\mu\geq \nu_\star$ and replace it with the sharper
\[ (1-\delta_\star) (\lambda\mu)^{1/2} \geq \max\left\{|(\kappa h)\id{t=0}|,\frac12 |(\iota h^{-1})\id{t=0}|\right\}.
\]
\end{Theorem}
It is important to notice that the above result holds with any $\tilde\lambda \in[1,\lambda\mu]$ but not with $\tilde\lambda=\lambda$ uniformly with respect to $\mu$ small. If it were the case, then the initial assumption on the high-order time derivatives of the unknown would be irrelevant as, by using the system of equations~\eqref{FG-adim2} and product estimates (see Lemma~\ref{L.products} below), we can estimate high-order time derivatives of $V$ from lower-order time derivatives, with a cost of powers of $\lambda^{1/2}$:
\[ \sum_{j=m+1}^{s} \lambda^{\frac{m-j}2}\norm{\partial_t^j V }_{H^{s-j}} \leq C\Big(h_\star,\nu,\sum_{j=0}^{m} \norm{\partial_t^j V}_{H^{s-j}}\Big)\times \sum_{j=0}^{m} \norm{\partial_t^j V}_{H^{s-j}}.\]
In particular, in the strong dispersion regime ($\mu\approx 1$), the explicit condition
\[\norm{V\id{t=0}}_{H^s}+\lambda^{1/2}\norm{\iota\id{t=0}}_{H^{s-1}}+\lambda^{1/2}\norm{\kappa\id{t=0}+\mu^{1/2}\nabla\cdot\u\id{t=0}}_{H^{s-1}}\leq M_0\]
is sufficient (applying Theorem~\ref{T.WP-mr} with $\tilde\lambda=\lambda\mu\approx\lambda$ and $m=1$) to guarantee the existence and uniform control of the corresponding solution ---but not its time derivatives--- on a time interval uniform with respect to $\lambda$ sufficiently large.
In the weak dispersion or shallow-water regime ($\mu\ll 1$), Theorem~\ref{T.WP-mr} relies on a strong constraint on the initial behavior of the solution, and it is natural to ask whether it is possible to provide initial data such that the corresponding solution satisfies the desired estimates for arbitrarily large $\lambda$ and small $\mu$, all the other parameters being fixed. We answer positively with the following result.

\begin{Theorem}\label{T.Prep}
Let $s,m\in\NN$, $s>d/2+1$, $s\geq m+1$ and $h_\star,M_0,\nu>0$. Then there exists $C_m,C_m'>0$ such that for any $(\lambda,\mu)\in\mathcal{S}_\nu$ and $\zeta_0,\u_0\in H^s(\DD^d)$ such that $h_0=1+\zeta_0\geq h_\star>0$ and
\[\norm{\zeta_0}_{H^s}+\norm{\u_0}_{H^s}\leq M_0,\]
the following holds. There exists $c^{(j)}\in H^s(\DD^d)$ for $j\in\{1,\cdots,m\}$ such that the unique strong solution to~\eqref{FG-adim} with initial data $U^{(m)}\id{t=0}=(\zeta_0,\u_0,\eta_0^{(m)},w_0^{(m)})$ where
\begin{equation}\label{def-id}
 w_0^{(m)}=\sum_{\substack{j\text{ odd} \\ 1\leq j\leq m}} \lambda^{-(j-1)/2} c^{(j)} \quad \text{ and }\quad  \eta_0^{(m)}=h_0+\sum_{\substack{j\text{ even}\\  2\leq j\leq m }} \lambda^{-j/2} c^{(j)} 
 \end{equation}
satisfies
\begin{equation}\label{est-prep-U}
\sum_{j=0}^{m+1}\norm{\partial_t^jU^{(m)}}_{H^{s-j}} (0)+\lambda \sum_{j=0}^m\norm{ \partial_t^j(\eta^{(m)}-h^{(m)})}_{H^{s-j}}(0)\leq C_m\, M_0.
\end{equation}
Moreover, we have for any $j\in\{1,\cdots,m\}$
\begin{equation}\label{est-prep-c}
\norm{c^{(j)}}_{H^{s-m}}+\mu^{\lfloor j/2\rfloor}\norm{c^{(j)}}_{H^{s}}\leq  C_m'\, M_0.
\end{equation}
We can choose $c^{(1)}=-h_0\nabla\cdot\u_0 $ and $c^{(2)}$ the unique solution to
\[\mathfrak{t}[h_0] c^{(2)}=h_0^3\left(\u_0\cdot \nabla (\nabla\cdot\u_0)-(\nabla\cdot\u_0)^2-\Delta \zeta_0-\nabla\cdot((\u_0\cdot\nabla)\u_0)\right)\]
where we define
\[\mathfrak{t}[h] \psi\eqdef \psi-\frac\mu3 h^3\nabla\cdot\left(h^{-1}\nabla \psi \right).\]
\end{Theorem}
\begin{Remark}
The expression for $c^{(2)}$ emerges when solving 
\[(h^2\ddot{\eta}-h^2\ddot{h})\id{t=0} =-\lambda(\eta-h)\id{t=0}-(h^2\ddot{h})\id{t=0}=\O(\lambda^{-1}).\]
The operator $\mathfrak{t}[h]$ is one-to-one and onto (see Lemma~\ref{L.c} below) if $\inf h>0$ and is in some sense conjugate to $\mathfrak{T}$ defined above, as 
\[\mathfrak{T}[h](h^{-1}\nabla \psi)=h^{-1}\nabla (\mathfrak{t}[h]\psi)\quad \text{ and } \quad h^3\nabla\cdot (\mathfrak{T}[h]\u)=\mathfrak{t}[h](h^3\nabla\cdot\u).\]
\end{Remark}
\begin{Remark}
A direct application of Theorems~\ref{T.WP-mr} and Theorem~\ref{T.Prep} shows that for any sufficiently regular initial data $(\zeta_0,\u_0)$ satisfying the non-cavitation assumption, one may associate a solution to~\eqref{FG-adim} satisfying the estimates of Theorem~\ref{T.consistency} uniformly with respect to $\mu$ possibly small and $\lambda$ sufficiently large, over the quasilinear time-scale (\ie inversely proportional to the size of the initial data).
Using energy estimates on the quasilinearized Green-Naghdi system (see~\cite{Li06,Alvarez-SamaniegoLannes08a,FujiwaraIguchi15,DucheneIsrawi}), we deduce that the difference between such solution and the exact solution to the Green-Naghdi system~\eqref{GN-adim} with the same initial data is of size $\O(\lambda^{-1}\mu t)$ on the quasilinear time-scale. This should be compared with the results of the previously mentioned works (and references therein) showing that the solution to the Green-Naghdi system is at a distance $\O(\mu^2 t)$ to the solution of the full water-waves system with corresponding initial data on the same timescale. Hence the Favrie-Gavrilyuk system produces as precise approximate solutions for long gravity waves as the Green-Naghdi system itself as soon as $\lambda\gtrsim \mu^{-1}$ and the initial data for $(\eta,w)$ is suitably chosen.
\end{Remark}

\subsection{Outline}

The remainder of this work is organized as follows. In Section~\ref{S.Strategy} we describe and comments our results and strategy together with relevant references in the literature.
In Section~\ref{S.notations}, we describe our notations. Section~\ref{S.technical} contains technical tools such as product estimates in Sobolev spaces and an elliptic estimate on the operator $\mathfrak{t}[h]$. We show in Section~\ref{S.Hyp} that the Favrie-Gavrilyuk system is hyperbolic under the usual non-cavitation assumption. We exhibit in Section~\ref{S.Sym} the symmetric structure of the system upon which our results are based.
Section~\ref{S.WP} contains the proof of Theorem~\ref{T.WP-mr}. Section~\ref{S.Prep} contains the proof of Theorem~\ref{T.Prep}. Section~\ref{S.conclusion} is dedicated to a summary and concluding remarks.

\subsection{Strategy}\label{S.Strategy}

As aforementioned, the main tool for proving the above results is energy estimates, which should hold uniformly with respect to the parameters $\lambda,\mu\in \mathcal{S}_\nu$. In order to obtain these estimates, we make use of a symmetric structure which is fairly easily deduced from the formulation~\eqref{FG-adim2}. As a matter of fact, we show in Section~\ref{S.Sym} below that the system can rewritten (when $d=2$) as
\[
 \S_t(V)\partial_t V+\S_x(V)\partial_x V+\S_y(V)\partial_y V=\lambda^{1/2}\J V+G(V),
\]
where $\S_t,\S_x,\S_y$ are smooth functions of $V$ with values into symmetric matrices, $\J$ is a skew-symmetric differential operator, and $G$ is a smooth function. Moreover $\S_t$ is positive definite in a hyperbolicity domain containing a neighborhood of the origin in $L^\infty(\DD^d)^{d+3}$.

We are obviously looking at a {\em singular limit} problem. Such problems, and in particular incompressible or low Mach number limits in the context of fluid mechanics, have a very rich history, which we shall not recall. We will only let the interested reader refer to, \eg, \cite{Gallagher05,Alazard08,Schochet05} for comprehensive reviews. Due to the non-trivial symmetrizer in front of the time derivative, the linearized system does not appear to be uniformly well-posed in Sobolev spaces as $\lambda\to \infty$ since small perturbations of the initial data might cause large changes in solutions. This is a noteworthy feature of the incompressible limit of the non-isentropic Euler equations, as studied in particular in~\cite{MetivierSchochet01}. However, our problem is different in nature as we do not wish to deal with large oscillations in time but rather aim at discarding them as spurious products of the approximation procedure. Hence we willingly restrict our study to {\em well-prepared initial data}, and as such our work is more directly related to pioneering works of Browning and Kreiss~\cite{BrowningKreiss82}, Klainerman and Majda~\cite{KlainermanMajda81}, and Schochet~\cite{Schochet86a}.
  In fact our proof of Theorem~\ref{T.WP-mr} closely follows the one of~\cite{Schochet86a}; while the proof of Theorem~\ref{T.Prep} is strongly inspired by~\cite{BrowningKreiss82}. However in both cases the proof requires significant adaptations in order to take into account the fact that the singular operator, {\em $\J$, is not homogeneous of order one}. 

The most serious novel difficulty stems from the fact that the contribution from order-zero terms in $\J$ are less well-behaved than order-one contributions, and that the latter are multiplied by a vanishing prefactor as $\mu\to 0$. This is the reason for the shortcoming described below Theorem~\ref{T.WP-mr}. We would like to explain now this discrepancy with the more standard setting ---studied in the previously mentioned references--- where $\J$ is homogeneous of order one. A toy model for the latter situation could be the following:
\[\partial_t u=\frac1\epsilon\, h \partial_x u\quad ; \quad \partial_t h=0.\]
Here $u$ is the singular variable while $h$ is a regular variable, given and independent of time.
The problem is reduced to a linear problem with variable coefficients, which is readily solvable by the methods of characteristics if we assume for instance that $h,u$ are initially regular and for any $x\in\RR$, $0<h_\star\leq h(x)\leq h^\star<\infty$. We see that variations of size $\delta$ in $h$ produce variations of size $1$ on $u$ at time $t=\epsilon/\delta$. However, the solution and its space-derivatives remain controlled for all times, uniformly with respect to $\epsilon$ small. 
This behavior is not shared for the toy model corresponding to $\J$ homogeneous of order zero, namely
\[\partial_t u=i\frac1\epsilon\,  h u \quad ; \quad \partial_t h=0.\]
The problem is now an ordinary differential equation in time where the space variable is a parameter. The solution $u(t,x)=u_0(x)\exp(i t h(x)/\epsilon)$ strongly oscillates with a different rate as $h(x)$ takes different values. Hence for positive times, the solution exhibits small scale oscillations, and space-derivatives are not uniformly controlled with respect to the parameter $\epsilon$ small. If variations of $h$ are of size $\delta$, it is necessary to prepare the initial data $u\id{t=0}=\O(\epsilon^m)$ in order to control $m$ space derivatives of the solution at time $t=1/\delta$. Our situation is roughly speaking a combination of the above where the size of $\mu$ measures the relative strength of the two influences. Based on the necessary properties satisfied by $\J$ (in particular Lemma~\ref{L.L} below) a toy model could be
\[\partial_t u=i\frac1\epsilon\, h \sqrt{1-\mu\partial_x^2}\, u \quad ; \quad \partial_t h=0.\]
Consistently with Theorem~\ref{T.WP-mr} ---and following the lines of its proof--- the assumption $u\id{t=0}=\O(\mu^{m/2})$ is sufficient to control $m$ space derivatives of the solution at time $t=1/\delta$, uniformly with respect to $\epsilon$ small.

\section{Preliminaries}

\subsection{Notations}\label{S.notations}

The parameter $d\in\{1,2\}$ denotes the horizontal space dimension, $X\in\DD^d$, where $\DD=\RR$ or $\DD=\TT$. If $d=2$, then we denote $X=(x,y)$. We sometimes assume for simplicity that $d=2$, the setting $d=1$ being recovered after straightforward simplifications. $\Id_{d}$ is the $d\times d$ identity matrix while ${\sf 0}_{d_1,d_2}$ is the $d_1\times d_2$ null matrix.

The notation $a\lesssim b$ means that 
 $a\leq C_0\ b$, where $C_0$ is a nonnegative constant whose exact expression is of no importance.
 We denote by $C(\lambda_1, \lambda_2,\dots)$ a nonnegative constant depending on the parameters
 $\lambda_1$, $\lambda_2$,\dots and whose dependence on the $\lambda_j$ is always assumed to be nondecreasing.

We use standard notations for functional spaces. $L^2(\DD^d)$ is the standard Hilbert space of square-integrable functions, associated with the inner-product  $\big(f_1,f_2\big)_{L^2}=\int_{\DD^d}f_1(x)f_2(x) \dd x$ and the norm $\norm{ f}_{L^2}=\big(\int_{\DD^d}\vert f(x)\vert^2 \dd x\big)^{1/2}$. The space $L^\infty(\DD^d)$ consists of all essentially bounded, Lebesgue-measurable functions
 $f$ with the norm
$\norm{f}_{L^\infty}= {\rm ess\,sup}_{x\in\DD^d} \vert f(x)\vert<\infty
$. We endow the space $W^{1,\infty}(\DD^d)=\{f, \mbox{ s.t. }   f\in L^{\infty}(\DD^d),\ \nabla f\in L^{\infty}(\DD^d)^d\}$ with its canonical norm. For any real constant $s\in\RR$, $H^s(\DD^d)$ denotes the Sobolev space of all tempered distributions $f$ with finite norm $\norm{ f}_{H^s}=\norm{ (1-\partial_x^2)^{s/2} f}_{L^2}$. For $j\in\NN$, $I$ a real interval and $X$ a normed space, $C^j(I;X)$ denotes the space of $X$-valued continuous functions on $[0,T)$ with continuous derivatives up to the order $j$. All these norms extend to multidimensional vector-valued functions by summing the contributions of all scalar components.

Additionally, we introduce non-standard norms, denoted with double-bars: $\Norm{\cdot}_{s,m,\tilde\lambda},\Norm{\cdot}_{s,m,\tilde\lambda,(1)},\Norm{\cdot}_{s,m,\tilde\lambda,(2)}$ defined in~\eqref{norm}--\eqref{norm-2} and $\Norm{\cdot}_{s,j}$ defined in~\eqref{norm-prep}.

When $\k=(k_x,k_y)\in\NN^d$ is a multi-index, $|\k|=k_x+k_y$ and $\partial^\k=\partial_x^{k_x}\partial_y^{k_y}$ when $d=2$ (otherwise $d=1$, $|\k|=k$ and $\partial^\k=\partial_X^k$). 
For $X,Y$ two closed linear operators (typically of differentiation and pointwise multiplication), we denote $[X,Y]\eqdef XY-YX $ the commutator whose domain is clear from the context.

\subsection{Technical tools}\label{S.technical}

We use mostly without reference the standard continuous Sobolev embedding $H^{s}(\DD^d) \subset L^\infty(\DD^d)$ for $s>d/2$ with
\[\norm{f}_{L^\infty}\lesssim \norm{f}_{H^s}.\]
The following product estimate is proved for instance in~\cite[Theorem~C.10]{Benzoni-GavageSerre07}.
\begin{Lemma}\label{L.product}
Let $f\in H^{s_1}(\DD^d)$ and $g\in H^{s_2}(\DD^d)$ and $s_1,s_2\geq s_0\geq 0 $ such that $s_1+s_2>s_0+d/2$. Then $fg\in H^{s_0}(\DD^d)$ and
\[\norm{fg}_{H^{s_0}}\lesssim \norm{f}_{H^{s_1}}\norm{g}_{H^{s_2}}.\]
\end{Lemma}
In particular, $H^s(\DD^d)$ is a Banach algebra as soon as $s>d/2$. We deduce by induction on the number of factors the following multilinear product estimate.
\begin{Lemma}\label{L.products}
Let $k\geq 2$ and $f_l\in H^{s_l}(\DD^d)$ for $l\in\{1,\dots,k\}$, with $s_l\geq s_0\geq 0$ and $\sum_{l=1}^k s_l>s_0+d/2$. Then $\prod_l f_l\in H^{s_0}(\DD^d)$ and
\[ \norm{\prod_l f_l}_{H^{s_0}}\lesssim \prod_l\norm{f_l}_{H^{s_l}}.\]
\end{Lemma}

We conclude with a technical result concerning the elliptic operator
\[\mathfrak{t}[h] :\psi\mapsto \psi-\frac\mu3 h^3\nabla\cdot\left(h^{-1}\nabla \psi \right).\]

\begin{Lemma}\label{L.c}
Let $s>1+d/2$ and $\zeta\in H^s(\DD^d)$ with be such that $1+\zeta\geq h_\star>0$. Then $\mathfrak{t}[h]:H^{1}\to H^{-1}$ is one-to-one and onto. Moreover, one has for any $\psi\in H^k(\DD^d)$ with $k\in\NN$ such that $k\leq s-1$,
\[\norm{\mathfrak{t}[h]^{-1}\psi}_{H^k}+\mu \norm{\mathfrak{t}[h]^{-1}\psi}_{H^{k+2}}\leq C(h_\star^{-1},\norm{\zeta}_{H^{s}})\norm{\psi}_{H^{k}}.\]
\end{Lemma}
\begin{proof}
The existence and uniqueness of $\varphi\in H^{k+2}(\DD^d)$ such that 
\begin{equation}\label{eq-c}
\mathfrak{t}[h]\varphi\eqdef \varphi-\frac{\mu h^3}3\nabla\cdot(h^{-1}\nabla\varphi)=\psi
\end{equation}
 follows from standard elliptic theory, and we focus on the estimates. Testing~\eqref{eq-c} against $h^{-3}\varphi$ yields
 \[\norm{\varphi}_{L^2}^2+\mu \norm{\nabla\varphi}_{L^2}^2\leq C(h_\star,\norm{h}_{L^\infty})\norm{\varphi}_{L^2}\norm{\psi}_{L^2}.\]
Using again~\eqref{eq-c}, we find
\[\mu\norm{\Delta\varphi}_{L^2} = \norm{\frac3{h^2}(\varphi-\psi)+\mu (h^{-1}\nabla h)\cdot\nabla\varphi}_{L^2}\leq C(h_\star^{-1},\norm{h}_{L^\infty},\mu^{1/2}\norm{\nabla h}_{L^\infty})\norm{\psi}_{L^2},\]
and the estimate is proved for $k=0$. For $1\leq k\leq s-1$, we differentiate~\eqref{eq-c} and find for any $\k$ such that $|\k|=k$,
\[h^{-3}\mathfrak{t}[h]\partial^\k\varphi=\partial^\k\psi-[\partial^\k,h^{-3}]\varphi+\frac\mu3 \nabla\cdot[\partial^\k,h^{-1}]\nabla\varphi.\]
Testing against $\partial^\k\varphi$, and using Lemma~\ref{L.products}, we have by induction on $k$
\[\norm{\varphi}_{H^k}+\mu^{1/2} \norm{\varphi}_{H^{k+1}}\leq C(h_\star^{-1},\norm{\zeta}_{H^{s}})\norm{\psi}_{H^{k}}.\]
and the result follows by using once again the identity and Lemma~\ref{L.products}.  
\end{proof}

\subsection{Hyperbolicity of the Favrie-Gavrilyuk system}\label{S.Hyp}

System~\eqref{FG-adim} is a quasilinear system of balance laws. It can be written under the matricial form (in dimension $d=2$) with $U\eqdef (\zeta,\u,\eta,w)$:
\[\partial_t U+\A_x(U)\partial_x U+\A_y(U)\partial_y U=F(U).\]
 Its principal symbol $\L\eqdef i\tau  +i\xi_x \A_x(U)+i\xi_y \A_y(U)$ is
 \[ \L=i \begin{pmatrix} 
 \Theta & h{\boldsymbol \xi}^\top & & \\
 \alpha {\boldsymbol \xi}& \Theta & \beta {\boldsymbol \xi} &  \\
 && \Theta &\\
 &&& \Theta
 \end{pmatrix}\]
where ${\boldsymbol \xi}=(\xi_x,\xi_y)^\top$, $\alpha=1+\mu\lambda\frac{\eta^2}{3h^3}$, $\beta=\frac{\mu\lambda}{3h}(1-\frac{2\eta}{h})$ and $\Theta=\tau +u_x\xi_x+u_y\xi_y$.
One immediately sees that $\Theta=0$ solves the characteristic equation, $\det L=0$, with multiplicity $d+1$ and corresponding eigenvectors
\[ (0,\xi_y,-\xi_x,0,0) \quad ; \quad (\beta,0,0,-\alpha,0) \quad ; \quad (0,0,0,0,1).\]
The first eigenvector corresponds to the evolution of the vorticity $\omega=\curl\u$ which is transported by the flow. The two other components are consistent with the fact that the phase velocity of the linearized Green-Naghdi system vanishes in the high-frequency limit (see~\cite{FavrieGavrilyuk17} for a comparative analysis of the dispersion relation of the Green-Naghdi and Favrie-Gavrilyuk systems).
There are two additional values of $\Theta$ solving $\det L=0$, namely $\Theta=\pm\sqrt{\alpha h}|{\boldsymbol \xi}|$ with corresponding eigenvectors
\[ (\mp\sqrt{h}|{\boldsymbol \xi}|,\sqrt{\alpha}\xi_x,\sqrt{\alpha}\xi_y,0,0).\]

Hence we see that the system is strongly hyperbolic as soon as one restricts to $U= (\zeta,\u,\eta,w)$ satisfying $\inf (1+\zeta)\geq h_\star>0$, as its principal symbol is smoothly diagonalizable. As a matter of fact the system is Friedrichs-symmetrizable, since one can exhibit a symmetrizer, $\S$, such that $\S\A_x$ and $\S\A_y$ are symmetric:
\[ \S\eqdef  \begin{pmatrix} 
 \alpha &     & \beta & \\
  &  h \Id_{d}    &  &  \\
\beta && \gamma &\\
 &&& 1
 \end{pmatrix}\]
where $\gamma$ is taken large enough in order to ensure that $\S$ is definite positive as soon as $\inf h\geq h_\star>0$.
From this, standard results yield Theorem~\ref{T.WP0}; see~\cite{Benzoni-GavageSerre07}.

\subsection{Symmetric structure of the Favrie-Gavrilyuk system}\label{S.Sym}

We can write the Favrie-Gavrilyuk system with variables $V=(\zeta,\u,\iota,\kappa)$, namely~\eqref{FG-adim2}, in a symmetric matricial form:
 \[\S_t(V)\big(\partial_t V+(\u\cdot\nabla )V\big)+\S_x(V)\partial_x V+\S_y(V)\partial_y V=\lambda^{1/2}\J V+G(V)\]
where
\[\S_t\eqdef  \begin{pmatrix} 
3\alpha \beta & 0 &  & \\
0  &  3h\beta\Id_{d} &  &  \\
 && h^{-1} &\frac{-\kappa h^2}{(\lambda\mu)^{1/2}}\\
&&\frac{-\kappa h^2}{(\lambda\mu)^{1/2}}& h^3
 \end{pmatrix}\]
 and
 \[\S_x\xi_x+\S_y\xi_y= \begin{pmatrix} 
  0 &3h\alpha\beta  {\bf \xi}^\top &  & \\
3h \alpha\beta  {\bf \xi}   &  {\sf 0}_{d,d} & \frac{\kappa^2h^2}{(\lambda\mu)^{1/2}} {\boldsymbol \xi}  &  \\
  &\frac{\kappa^2h^2}{(\lambda\mu)^{1/2}} {\boldsymbol \xi}^\top & 0& 0 \\
 &   & 0 & 0
  \end{pmatrix} ,\
  G(V)= \begin{pmatrix} 0\\
{\bf 0}\\
\mu^{-1/2}h^{-1}\kappa\iota \\
-\mu^{-1/2}h^3\kappa^2 \end{pmatrix}.\]
with (misusing notation with the preceding section) $\alpha=1+\frac{\iota^2}{3h^2}$, $\beta=\frac{1-\frac{\kappa^2 h^2}{\lambda\mu}}{1+\frac{2\iota}{(\lambda\mu)^{1/2} h}}$.

\section{Large time well-posedness}\label{S.WP}

In this section, we provide uniform energy estimates satisfied by well-prepared strong solutions of the Favrie-Gavrilyuk system~\eqref{FG-adim}, which yield the large time well-posedness result of Theorem~\ref{T.WP-mr}. In the spirit of~\cite{Schochet86a}, we define for $m,s\in\NN$, $1\leq m\leq s$, $\tilde \lambda\in(0,+\infty)$ and sufficiently regular functions $V$
\begin{align}
\label{norm}
\Norm{V}_{s,m,\tilde\lambda}^2&\eqdef \sum_{j=0}^m\norm{\partial_t^j V}_{H^{s-j}}^2+\sum_{j=m+1}^{s} \tilde\lambda^{m-j}\norm{\partial_t^j V }_{H^{s-j}}^2,\\
\label{norm-1}
\Norm{V}_{s,m,\tilde\lambda,(1)}^2&\eqdef \sum_{j=0}^{m-1}\sum_{|\k|=0}^{ s-j}\big(\S_t(V)\partial_t^j \partial^\k V,\partial_t^j \partial^\k V\big)_{L^2}\nonumber\\
&\qquad +\sum_{j=m}^s \tilde\lambda^{m-j} \big(\S_t(V)\partial_t^j V,\partial_t^j V\big)_{L^2},\\
\label{norm-2}
\Norm{V}_{s,m,\tilde\lambda,(2)}^2&\eqdef \sum_{j=m}^{s-1}\sum_{|\k|=1}^{s-j}  \tilde\lambda^{m-j} \norm{\partial_t^j \partial^\k V}_{L^2}^2.
\end{align}
By convention $\Norm{V}_{s,m,\tilde\lambda,(2)}=0$ if $m=s$.
Of course the notation in~\eqref{norm-1} is abusive as the right-hand side does not define a norm. However the existence of regular solutions provided by Theorem~\ref{T.WP0} allows to bypass the usual step of the linear Cauchy problem with variable coefficients and we consider directly the fully nonlinear problem. Below we consider $V=(\zeta,\u,\iota,\kappa)\in C^0([0,T];H^s(\DD^d)^{d+3})$ strong solution to~\eqref{FG-adim2}, hence we have $V\in\bigcap_{j=0}^s C^j([0,T];H^{s-j}(\DD^d)^{d+3}) $.

In this section, we fix $\lambda,\mu\in  (0,+\infty)$ and assume for simplicity that 
\[\lambda\geq 1 \quad ; \quad \mu\leq 1 \quad ; \quad \lambda\mu\geq 1.\]
Hence $(\lambda,\mu)\in\mathcal{S}_1$, recalling notation~\eqref{not-Snu}; the only change in the generally case $(\lambda,\mu)\in \mathcal{S}_\nu$ with $\nu>0$ is that all constants then depend on the parameter $\nu$.

Our results rely on the following estimates satisfied by strong solutions to~\eqref{FG-adim2} satisfying reasonable hyperbolicity conditions.
\begin{Proposition}\label{P.equiv}
Let $s\in\NN$ with $s>1+d/2$ and $h_\star,h^\star>0$, $\delta_\star\in(0,1)$.  There exists $C_1=C(h_\star^{-1},\delta_\star^{-1},h^\star)$ such that for any strong solution to~\eqref{FG-adim2}, $V=(\zeta,\u,\iota,\kappa)\in C^0([0,T];H^s(\DD^d)^{d+3})$ satisfying, uniformly on $[0,T]\times\DD^d$, $ h\eqdef 1+\zeta \in [h_\star,h^\star] $, 
\begin{equation}\label{eq-hyp}
 h|\kappa |\leq (1-\delta_\star) (\lambda\mu)^{1/2}\quad \text{ and } \quad 2 h^{-1}|\iota |\leq (1-\delta_\star) (\lambda\mu)^{1/2},
 \end{equation}
for any $m\in\NN$ such that $1\leq m\leq s$ and for any $\tilde\lambda\in(0,+\infty)$,  one has
\begin{equation}\label{est-equiv}
\frac1{C_1} \Norm{V}_{s,m,\tilde\lambda} \leq \Norm{V}_{s,m,\tilde\lambda,(1)}+\Norm{V}_{s,m,\tilde\lambda,(2)}\leq C_1 \Norm{V}_{s,m,\tilde\lambda}
\end{equation}
uniformly on $t\in[0,T]$.
\end{Proposition}
\begin{Proposition}\label{P.energy-estimate}
Let $s\in\NN$ with $s>1+d/2$ and $h_\star,M>0$ and $\delta_\star\in(0,1)$. There exists $C_2=C(h_\star^{-1},\delta_\star^{-1},M)$ such that for any strong solution to~\eqref{FG-adim2}, $V=(\zeta,\u,\iota,\kappa)\in C^0([0,T];H^{s+1}(\DD^d)^{d+3})$ satisfying the assumptions of Proposition~\ref{P.equiv} uniformly on $[0,T]\times\DD^d$ and $\sup_{t\in[0,T]}\Norm{V}_{s,m,\tilde\lambda}\leq M$, for any $m\in\NN$ such that $1\leq m\leq s$ and for any $\tilde\lambda\in [1,+\infty)$, one has
\begin{equation}\label{est-energy}
\frac{\dd}{\dd t}\Norm{V}_{s,m,\tilde\lambda,(1)}^2 \leq C_2 \Norm{V}_{s,m,\tilde\lambda}^3 
\end{equation}
uniformly on $t\in[0,T]$.
\end{Proposition}
\begin{Proposition}\label{P.bootstrap}
Let $s\in\NN$ with $s>1+d/2$ and $h_\star,M,M_{(1)}>0$. There exists $\nu_\star=C(h_\star^{-1},M)>0$ and $C_3=C(h_\star^{-1},M_{(1)})$ such that for any strong solution to~\eqref{FG-adim2}, $V=(\zeta,\u,\iota,\kappa)\in C^0([0,T];H^{s+1}(\DD^d)^{d+3})$ satisfying $h=1+\zeta\geq h_\star>0$ uniformly on $[0,T]\times\DD^d$, $\sup_{t\in[0,T]}\Norm{V}_{s,m,\tilde\lambda}\leq M$, ${\sup_{t\in[0,T]}\Norm{V}_{s,m,\tilde\lambda,(1)}\leq M_{(1)}}$,
for any $m\in\NN$ such that $1\leq m\leq s$ and any $\tilde\lambda$ such that $\lambda\mu\geq \tilde\lambda \geq \nu_\star$, one has
\begin{equation}\label{est-bootstrap}
\Norm{V}_{s,m,\tilde\lambda ,(2)}\leq  C_3\Norm{V}_{s,m,\tilde\lambda,(1)} 
\end{equation}
uniformly on $t\in[0,T]$.
\end{Proposition}

The fact that~\eqref{est-bootstrap} holds with $\tilde\lambda=\lambda\mu$ but not with $\tilde\lambda=\lambda$ (uniformly with respect to $\mu$), which itself can be tracked down to the lack of uniformity in Lemma~\ref{L.L} below, is the reason why we cannot obtain uniform bounds on solutions without a fine preparation on the initial data; see the discussion below Lemma~\ref{T.WP-mr}. The proof of Proposition~\ref{P.equiv} is an exercise using the explicit formula for $\S_t$ given in Section~\ref{S.Sym}. We postpone the proof of Propositions~\ref{P.energy-estimate} and~\ref{P.bootstrap} to Sections~\ref{S.energy-estimate} and~\ref{S.bootstrap} (respectively), and complete the proof of Theorem~\ref{T.WP-mr} below. 

\subsection{Proof of Theorem~\ref{T.WP-mr}}

Let us first assume that the initial data $V_0\in H^{s+1}(\DD^d)$, so that by Theorem~\ref{T.WP0} ---and Lemma~\ref{L.products} to handle the nonlinear change of variables~\eqref{change}--- we have 
\[V=(\zeta,\u,\iota,\kappa)\in \bigcap_{j=0}^{s}C^{j+1}([0,T^\star);H^{s-j}(\DD^d)^{d+3})\]
and hence all the ``norms'' below are well-defined and differentiable on $t\in[0,T^\star)$.
We fix $m\in\NN$ with $1\leq m\leq s$, denote $M_{(1)}\eqdef \Norm{V}_{s,m,\tilde\lambda,(1)}(0) $ and
\begin{multline*}
T^\sharp\eqdef \sup\big\{ t\geq 0 \text{ such that } \Norm{V}_{s,m,\tilde\lambda,(1)}(t)\leq 2M_{(1)}  \text{ and }\\
 \forall X\in\DD^d, \quad h_\star/2\leq 1+\zeta(t,X)\leq 2\norm{h(0,\cdot)}_{L^\infty}\big\}.
\end{multline*}
By a continuity argument, we have that $T^\star>0$.
Propositions~\ref{P.equiv},~\ref{P.energy-estimate} and~\ref{P.bootstrap} yield $\nu_\star=C(2h_\star^{-1},M)>0$ such that
for any $\tilde\lambda$ such that $\lambda\mu\geq\tilde\lambda\geq \nu_\star$ and provided $\sup_{t\in[0,T^\sharp]} \Norm{V}_{s,m,\tilde\lambda}\leq M$, then~\eqref{eq-hyp} holds with $\delta_\star=1/2$ and one has for any $t\in[0,T^\sharp]$:
\begin{align*}\Norm{V}_{s,m,\tilde\lambda}&\leq C_0 \Norm{V}_{s,m,\tilde\lambda,(1)},\\
\Norm{V}_{s,m,\tilde\lambda,(1)}&\leq C_1 \Norm{V}_{s,m,\tilde\lambda}
\end{align*}
with $C_1=C(2h_\star^{-1},2\norm{h(0,\cdot)}_{L^\infty})$, $C_3= C(2h_\star^{-1},2  M_{(1)})$, $C_0=C_1(1+C_3)$; and
\[\frac{\dd}{\dd t}\Norm{V}_{s,m,\tilde\lambda,(1)}^2 \leq  C_2 C_0^3\Norm{V}_{s,m,\tilde\lambda,(1)}^3 \]
with $C_2=C(2h_\star^{-1},M)$, from which we deduce
\[\Norm{V}_{s,m,\tilde\lambda,(1)}\leq M_{(1)} \exp(M_{(1)} C_2 C_0^3 t).\]
At time $t=0$, we have $\norm{h}_{L^\infty}(0)\lesssim M_0$ and $M_{(1)}\leq C_1 M_0$ and we may set above $M=2 C_0M_{(1)}\leq 2C_0 C_1 M_0=M_0C(h_\star^{-1},M_0)$.
We also have by the continuous Sobolev embedding $H^{s-1}\subset L^\infty$ that there exists $c_s>0$ such that for any $t\in[0,T^\sharp]$,
\[1+\zeta(t,x)=1+\zeta(0,x)+\int_0^t \partial_t\zeta(s,x)\dd s\in [h_\star - M c_s t, \norm{h(0,\cdot)}_{L^\infty}+M c_s t].\]
Hence, we deduce by continuity and from the above that $T^\star> T^\sharp\geq (M_{(1)}\tau )^{-1}$ with $\tau= \sup\{ 4 C_0 c_s h_\star^{-1}, C_2C_0^3/\ln 2\}= C(h_\star^{-1},M_{(1)})$, which completes the proof when the initial data $V_0\in H^{s+1}(\DD^d)$. The general case $V_0\in H^{s}(\DD^d)$ is deduced by a standard regularization and compactness argument; see for instance~\cite[pp.~1631-1632]{Schochet86a}.

The improved result in the setting $m=s$ is proved in the same way, using that Propositions~\ref{P.equiv} and~\ref{P.energy-estimate} alone are sufficient to have the necessary estimates and that the initial assumption~\eqref{eq-hyp} propagates (replacing $\delta_\star$ with $\delta_\star/2$) on the quasilinear timescale since
$\norm{V(t)-V(0)}_{L^\infty}\leq t\norm{\partial_t V}_{L^\infty}\lesssim M c_s t$.

\subsection{Energy estimates; proof of Proposition~\ref{P.energy-estimate}}\label{S.energy-estimate}

Here and in the following, we denote $V=(\zeta,\u,\iota,\kappa)\in C^0([0,T];H^{s+1}(\DD^d)) \cap C^1([0,T];H^s(\DD^d))$ a strong solution to~\eqref{FG-adim2} satisfying $h=1+\zeta\geq h_\star>0$. By applying iteratively the equation, one has $\partial_t^j V\in  C^1([0,T];H^{s-j}(\DD^d))$ and hence all the terms below are well-defined and continuous with respect to time. Recall (see Section~\ref{S.Sym}) that~\eqref{FG-adim2} has the following form
 \begin{equation} \label{FG-sym}
 \S_t(V)\partial_t V+\S_x(V)\partial_x V+\S_y(V)\partial_y V=\lambda^{1/2}\J V+G(V),
 \end{equation}
 where $\S_t,\S_x,\S_y$ are smooth functions of $V$ with values into symmetric matrices (we simply denote $\S_t,\S_x,\S_y$ for $\S_t(V),\S_x(V),\S_y(V)$ for the sake of conciseness below), $\J$ is skew-symmetric, and $G$ is a smooth function.
We prove below estimate~\eqref{est-energy} by standard energy method, differentiating~\eqref{FG-sym} and testing against derivatives of $V$.

By testing~\eqref{FG-sym} against $V$ and using the symmetry of $\S_t,\S_x,\S_y$ and the skew-symmetry of $\J$, we find
\begin{multline*}
\frac12\frac{\dd}{\dd t}\big(\S_t V,V\big)_{L^2}=\frac12 \big([\partial_t,\S_t]V,V\big)_{L^2}+\frac12\big([\partial_x,\S_x] V,V\big)_{L^2}+\frac12\big([\partial_y,\S_y] V,V\big)_{L^2}\\+\big(G(V), V\big)_{L^2}.
\end{multline*}
It follows immediately by continuous Sobolev embedding $H^{s-1}\subset L^\infty$ for any $s>1+d/2$ that
\begin{multline*}\frac12\frac{\dd}{\dd t}\big(\S_t V,V\big)_{L^2}\leq C(h_\star^{-1},\norm{V}_{H^s})\big(  \norm{\partial_t V}_{H^{s-1}}+\norm{ V}_{H^s}+\mu^{-1/2}\norm{\kappa}_{L^\infty}\big)\norm{V}_{L^2}^2.
\end{multline*}

We now control space derivatives of the solution. Given $\k=(k_x,k_y)$ such that $|\k|\leq s$, we apply $\partial^\k=\partial_x^{k_x}\partial_y^{k_y}$ to~\eqref{FG-sym} and test against $\partial^\k V$. Because $\J$ commutes with space derivatives, we have
\begin{multline*}
\frac12\frac{\dd}{\dd t}\big(\S_t \partial^\k V, \partial^\k V\big)_{L^2}=\frac12 \big(\big([\partial_t,\S_t]+[\partial_x,\S_x] +[\partial_y,\S_y] \big)\partial^\k V,\partial^\k V\big)_{L^2}\\+\big([\partial^\k, \S_t]\partial_t V+[\partial^\k, \S_x]\partial_xV+[\partial^\k, \S_y]\partial_yV, \partial^\k V\big)_{L^2}+\big(\partial^\k G(V), \partial^\k V\big)_{L^2}.
\end{multline*}
The first component is estimated as above and we have
\begin{multline*}\norm{[\partial_t,\S_t] \partial^\k V+[\partial_x,\S_x] \partial^\k V+[\partial_y,\S_y] \partial^\k V}_{L^2}
\\
\leq C(h_\star^{-1},\norm{V}_{H^s})\big( \norm{\partial_t V}_{H^{s-1}}+\norm{ V}_{H^s}\big)\norm{\partial^\k V}_{L^2}.
\end{multline*}
 Using Lemma~\ref{L.products}, we find
\begin{align*}\norm{[\partial^\k, \S_t]\partial_t V}_{L^2}&\leq C(h_\star^{-1},\norm{V}_{H^s}) \norm{V}_{H^s}\norm{\partial_t V}_{H^{s-1}},\\
\norm{[\partial^\k, \S_x]\partial_x V}_{L^2}+\norm{[\partial^\k, \S_y]\partial_y V}_{L^2}&\leq C(h_\star^{-1},\norm{V}_{H^s})\norm{ V}_{H^s}^2,\\
\norm{\partial^\k G(V)}_{L^2}&\leq C(h_\star^{-1},\norm{V}_{H^s}) \big(\mu^{-1/2}\norm{\iota}_{H^{s-1}}+\mu^{-1/2}\norm{\kappa}_{H^{s-1}}\big)\norm{V}_{H^s}.
\end{align*}
Notice that by using the last two equations of~\eqref{FG-adim2} and since $\lambda\mu\geq 1$, we have
\[\mu^{-1/2}\norm{\iota}_{H^{s-1}}+\mu^{-1/2}\norm{\kappa}_{H^{s-1}}\leq C(h_\star^{-1},\norm{V}_{H^s}) (\norm{\partial_t V}_{H^{s-1}}+\norm{V}_{H^s}).\]
Altogether, and applying Cauchy-Schwarz inequality, we proved
\[\frac{\dd}{\dd t}\big(\S_t \partial^\k V,\partial^\k V\big)_{L^2}\leq C(h_\star^{-1},\Norm{V}_{s,1,\tilde\lambda})\Norm{V}_{s,1,\tilde\lambda}^3.\]

The control of the first $m-1$ time-derivatives of the solution is identical, using that $m$ time-derivatives are uniformly controlled by $\Norm{V}_{s,m,\tilde\lambda}$; hence we have 
\[\sum_{j=0}^{m-1}\sum_{|\k|=0}^{s-j}\frac{\dd}{\dd t}\big(\S_t \partial_t^j\partial^\k V,\partial_t^j\partial^\k V\big)_{L^2}\leq C(h_\star^{-1},\Norm{V}_{s,m,\tilde\lambda})\Norm{V}_{s,m,\tilde\lambda}^3.\]

This estimate cannot be straightforwardly pushed towards higher time derivatives, in particular due to the lack of uniform estimate for 
\[\norm{[\partial_t^j\partial^\k, \S_t]\partial_t V}_{L^2}\]
when $j\geq m$. However we see below that the desired uniform estimate does hold when $\k=0$. Proceeding as above, we have
\begin{multline*}
\frac12\frac{\dd}{\dd t}\big(\S_t \partial_t^j V, \partial_t^j V\big)_{L^2}=\frac12 \big([\partial_t,\S_t] \partial_t^j V+[\partial_x,\S_x] \partial_t^j V+[\partial_y,\S_y] \partial_t^j V,\partial_t^j V\big)_{L^2}\\+\big([\partial_t^j, \S_t]\partial_t V+[\partial_t^j, \S_x]\partial_xV+[\partial_t^j, \S_y]\partial_yV, \partial_t^j V\big)_{L^2}+\big(\partial_t^j G(V), \partial_t^j V\big)_{L^2}.
\end{multline*}
The first terms of the right-hand side are estimated as above:
\begin{multline*}\norm{[\partial_t,\S_t] \partial_t^j V+[\partial_x,\S_x] \partial_t^j V+[\partial_y,\S_y] \partial_t^j V}_{L^2}
\\
\leq C(h_\star^{-1},\norm{V}_{H^s})\big( \norm{\partial_t V}_{H^{s-1}}+\norm{ V}_{H^s}\big)\norm{\partial_t^j V}_{L^2}.
\end{multline*}
The other terms require the use of Lemma~\ref{L.products} and to pay attention to powers of $\tilde\lambda$. Taking advantage of a gain of a factor $\tilde\lambda^{-m/2}$ as soon as time derivatives are distributed, we find that for any $\tilde\lambda\geq 1$, and any $j\geq m\geq 1$,
\[
\tilde\lambda^{\frac{m-j}2}\norm{[\partial_t^j, \S_t]\partial_t V+[\partial_t^j, \S_x]\partial_x V+[\partial_t^j, \S_y]\partial_y V}_{L^2}\leq C(h_\star^{-1},\Norm{V}_{s,m,\tilde\lambda}) \Norm{V}_{s,m,\tilde\lambda}^2.
\]
Finally, one obtains similarly as above
\begin{align*}\tilde\lambda^{\frac{m-j}2}\norm{\partial_t^j G(V)}_{L^2}&\leq C(h_\star^{-1},\Norm{V}_{s,m,\tilde\lambda}) \mu^{-1/2}\big(\Norm{\kappa}_{s-1,m,\tilde\lambda}+\Norm{\iota}_{s-1,m,\tilde\lambda}) \Norm{V}_{s,m,\tilde\lambda}\\
&\leq  C(h_\star^{-1},\Norm{V}_{s,m,\tilde\lambda})\Norm{V}_{s,m,\tilde\lambda}.
\end{align*}
Altogether, we proved
\[
\frac{\dd}{\dd t}\big(\S_t \partial_t^j V, \partial_t^j V\big)_{L^2}
\leq  C(h_\star^{-1},\Norm{V}_{s,m,\tilde\lambda})\Norm{V}_{s,m,\tilde\lambda}^3,\]
for any $j\in \{m,\dots,s\}$. This completes the proof of Proposition~\ref{P.energy-estimate}.

\subsection{Filling in estimates; proof of Proposition~\ref{P.bootstrap}}\label{S.bootstrap}

This section is dedicated to the proof of Proposition~\ref{P.bootstrap}. Contrarily to Proposition~\eqref{P.energy-estimate}, we shall rely strongly on properties of $\J$. Recall our system is of the form~\eqref{FG-sym} with
\[\J= \begin{pmatrix} 
  0 & {\sf 0}_{1,d}  &  & \\
 {\sf 0}_{d,1} &  {\sf 0}_{d,d} &  \mu^{1/2}\nabla   &\\
  &\mu^{1/2}\nabla^\top & 0& 1 \\
 &  &-1 & 0
  \end{pmatrix}.\]
  We introduce $\Pi^\r$ and $\Pi^\s$ the projections onto the kernel and non-zero eigenvalues of $\J$: 
  \[\Pi^\r\eqdef \begin{pmatrix} 
    1 &  {\sf 0}_{1,d} & 0 & 0\\
   {\sf 0}_{d,1}  & \Id_{d}+ \frac{\mu\nabla\nabla^\top}{1-\mu\Delta}   &  {\sf 0}_{d,1} & \frac{\mu^{1/2}\nabla}{1-\mu\Delta} \\
   0 &  {\sf 0}_{1,d} & 0& 0 \\
  0 &-\frac{\mu^{1/2}\nabla^\top}{1-\mu\Delta} &0 & \frac{-\mu\Delta}{1-\mu\Delta}
    \end{pmatrix}, \]
    and
    \[
    \Pi^\s\eqdef 
    \begin{pmatrix} 
        0 &  {\sf 0}_{1,d} & 0 & 0\\
       {\sf 0}_{d,1}  & -\frac{\mu\nabla\nabla^\top}{1-\mu\Delta}   &  {\sf 0}_{d,1} & -\frac{\mu^{1/2}\nabla}{1-\mu\Delta} \\
       0 &  {\sf 0}_{1,d} & 1& 0 \\
      0 &\frac{\mu^{1/2}\nabla^\top}{1-\mu\Delta} &0 & \frac{1}{1-\mu\Delta}
        \end{pmatrix} .
   \]
        By definition, we have the following properties
        \begin{align*}
      &(\Pi^\s)^2=\Pi^\s, &  & (\Pi^\r)^2=\Pi^\r\\
       & \Pi^\s+\Pi^\r=\Id_{d+3}, & & \Pi^\s\Pi^\r={\sf 0}_{d+3,d+3}\\
 &\Pi^\s \ \J\ =\ \J\ \Pi^\s \ = \ \J, & & \Pi^\r \ \J\ =\ \J\ \Pi^\r \ = \ {\sf 0}_{d+3,d+3} . 
 \end{align*}
  Thanks to the skew-symmetry of $\J$ and the property that the number of nonzero eigenvalues of its symbol does not depend on the (non-zero) frequency (there are always two non-zero eigenvalues and the kernel dimension is $d$+1), we have that $\Pi^\s$ and $\Pi^\r$ are bounded symmetric operators acting on the Hilbert space $L^2(\DD^d)^{d+3}$: for any $U,V\in L^2(\DD^d)^{d+3}$,
  \[\big(\Pi^{\substack{\s \\ \r}} U,V\big)_{L^2}=\big( U,\Pi^{\substack{\s \\ \r}} V\big)_{L^2} \quad \text{ and } \quad  \norm{V}_{L^2}^2=\norm{\Pi^\r V}_{L^2}^2+\norm{\Pi^\s V}_{L^2}^2.\]
  In the following, we denote $V^\r\eqdef\Pi^\r V$, $V^\s\eqdef \Pi^\s V$. Using that $ \Pi^{\substack{\s \\ \r}}$ commutes with space and time derivatives, we deduce from the above that
\[\Norm{V}_{s,m,\tilde\lambda,(2)}^2\leq  \sum_{j=m}^{s-1}\tilde\lambda^{m-j}\norm{\partial_t^j V^\r}_{H^{s-j}}^2+\tilde\lambda^{m-j}\norm{\partial_t^j  V^\s}_{H^{s-j}}^2.\]
We provide in the following sections estimates for 
\[ N^\r_{j,k,m,\tilde\lambda}\eqdef \tilde\lambda^{\frac{m-j}2}\norm{\partial_t^j V^\r}_{H^k}\quad \text{ and } \quad  N^\s_{j,k,m,\tilde\lambda}\eqdef\tilde\lambda^{\frac{m-j}2}\norm{\partial_t^j V^\s}_{H^k}.\]  The main tool for estimating $N^\s_{m,k,j,\tilde\lambda}$ is that, when restricting to the singular subspace, $\J$ is a homeomorphism from $H^k$ to $H^{k-1}$.
 \begin{Lemma}\label{L.L}
 Let $k\in \RR$ and $U\in H^{k-1}(\DD^d)^{d+3}$ such that $U=\Pi^\s U$. Then there exists a unique $V\in H^{k}(\DD^d)^{d+3}$ such that $V=\Pi^\s V$ and $U=\J V$. Moreover, one has $V=\frac{-\J}{1-\mu\Delta} U$ and in particular there exists $C_{\J}$ such that
 \[\norm{V}_{H^{k}} \leq C_{\J}\mu^{-1/2}\norm{U}_{H^{k-1}}.\]
 \end{Lemma}

\subsubsection{Estimate of the singular contribution, $N^\s_{j,k,m,\tilde\lambda}$}

Differentiating with time the system~\eqref{FG-sym}, and projecting onto the singular subspace yields the identity for any $j\in\NN$:
 \[\Pi^\s \partial_t^j \left(\S_t(V)\partial_t V+\S_x(V)\partial_x  V+\S_y(V)\partial_y V-G(V)\right)=\lambda^{1/2} \Pi^\s\J \Pi^\s\partial_t^j V.\]
By distributing the time derivatives, paying attention to powers of $\tilde\lambda$ and using Lemma~\ref{L.products}, we find that for any $0\leq j\leq s$ and $k\in \NN$ such that $k\leq s-j$: 
\begin{multline*}
\tilde\lambda^{\frac{m-j}2}\norm{\partial_t^j \big(\S_t(V)\partial_t V\big)}_{H^{k-1}}\leq C(h_\star^{-1},\norm{V}_{L^\infty}) \tilde\lambda^{\frac{m-j}2}\norm{\partial_t^{j+1} V}_{H^{k-1}}\\
+C(h_\star^{-1},\Norm{V}_{s,m,\tilde\lambda})\Norm{V}_{s,m,\tilde\lambda}^2.
\end{multline*}
Similarly,
\begin{multline*}
\tilde\lambda^{\frac{m-j}2}\norm{\partial_t^j \big(\S_x(V)\partial_x V\big)}_{H^{k-1}}\leq C(h_\star^{-1},\norm{V}_{L^\infty}) \tilde\lambda^{\frac{m-j}2}\norm{\partial_t^{j} V}_{H^{k}}\\
+C(h_\star^{-1},\Norm{V}_{s,m,\tilde\lambda})\Norm{V}_{s,m,\tilde\lambda}^2.
\end{multline*}
 Finally,
\begin{align*}\tilde\lambda^{\frac{m-j}2}\norm{\partial_t^j G(V)}_{H^{k-1}}&\leq   C(h_\star^{-1},\Norm{V}_{s,m,\tilde\lambda}) \mu^{-1/2}\big(\Norm{\kappa}_{s-1,m,\tilde\lambda}+\Norm{\iota}_{s-1,m,\tilde\lambda}\big) \Norm{V}_{s,m,\tilde\lambda}\\
&\leq   C(h_\star^{-1},\Norm{V}_{s,m,\tilde\lambda})  \Norm{V}_{s,m,\tilde\lambda}^2.
\end{align*}
 Altogether and using Lemma~\ref{L.L}  we deduce
\begin{multline}\label{est-singular}
N^\s_{j,k,m,\tilde\lambda}\leq C_{\J}(\lambda\mu)^{-\frac12}  C(h_\star^{-1},\Norm{V}_{s,m,\tilde\lambda,(1)}) \big(\tilde\lambda^{\frac12} N^\s_{j+1,k-1,m,\tilde\lambda}+\tilde\lambda^{\frac12} N^\r_{j+1,k-1,m,\tilde\lambda}  \big)\\
+C_{\J}(\lambda\mu)^{-\frac12} C(h_\star^{-1},\Norm{V}_{s,m,\tilde\lambda})\Norm{V}_{s,m,\tilde\lambda},
\end{multline}
where we recall that $C_{\J}>0$ is defined in Lemma~\ref{L.L}.

\subsubsection{Estimate of the regular component $N^\r_{j,k,m,\lambda}$}

Now we project system~\eqref{FG-sym} onto the regular subspace and apply the differential operator $\partial^\k\partial_t^{j-1}$ for $1\leq m\leq j\leq s$ and $\k\in \NN^d$ such that $|\k|\leq s-j$. Testing against $\Pi^\r \partial^\k \partial_t^{j}V$ yields
\[\big(  \partial^\k \partial_t^{j-1} \left(\S_t\partial_t V+\S_x\partial_x  V+\S_y\partial_y V-G(V)\right),\Pi^\r \partial^\k \partial_t^{j}V\big)_{L^2}=0,\]
which we decompose as follows:
\begin{multline*}
\big(\S_t\Pi^\r \partial^\k \partial_t^{j}V,\Pi^\r \partial^\k \partial_t^{j}V\big)_{L^2}+\big(\S_t\Pi^\s \partial^\k \partial_t^{j}V,\Pi^\r \partial^\k \partial_t^{j}V\big)_{L^2}\\
+\big(\partial^\k  \partial_t^{j-1}\big(\S_x \partial_xV+\S_y\partial_yV\big),\Pi^\r \partial^\k \partial_t^{j}V\big)_{L^2}\\
+\big( \partial^\k [\partial_t^{j-1} ,\S_t]\partial_t V,\Pi^\r \partial^\k \partial_t^{j}V\big)_{L^2}\\
+\big( [\partial^\k  ,\S_t]\partial_t^j V,\Pi^\r \partial^\k \partial_t^{j}V\big)_{L^2}\\
+\big(  \partial^\k \partial_t^{j-1}G(V),\Pi^\r \partial^\k \partial_t^{j}V\big)_{L^2}=0.
\end{multline*}
Under the assumptions of Proposition~\ref{P.equiv}, the first contribution gives us the desired control
\[|N^\r_{j,k,m,\tilde\lambda}|^2 \leq  \tilde\lambda^{m-j} C(h_\star^{-1},\delta_\star^{-1},\norm{V}_{L^\infty})\sum_{|\k|=0}^k\big(\S_t\Pi^\r \partial^\k \partial_t^{j}V,\Pi^\r \partial^\k \partial_t^{j}V\big)_{L^2},\]
and the second contribution is estimated through
\[\tilde\lambda^{\frac{m-j}2} \norm{\S_t\Pi^\s \partial^\k \partial_t^{j}V}_{L^2}\leq  C(h_\star^{-1},\norm{V}_{L^\infty}) N^\s_{j,k,m,\tilde\lambda}.\]
As for the second line, we estimate differently depending on the value of $j$. If $j\geq m+1$, we use the gain of the prefactor $\tilde\lambda^{-1/2}$ stemming from the fact that only $j-1$ time derivatives are involved:
\[\tilde\lambda^{\frac{m-j}2}\norm{\partial_t^{j-1}\big(\S_x\partial_xV+\S_y\partial_yV\big)}_{H^k}\leq \tilde\lambda^{\frac{-1}2}C(h_\star^{-1},\Norm{V}_{s,m,\tilde\lambda}) \Norm{V}_{s,m,\tilde\lambda}.\]
When $j=m$, we do not have the gain of the prefactor $\tilde\lambda^{-1/2}$ but less than $m-1$ time derivatives are involved:
\[\tilde\lambda^{\frac{m-j}2}\norm{\partial_t^{j-1}\big(\S_x\partial_xV+\S_y\partial_yV\big)}_{H^k}\leq C(h_\star^{-1},\Norm{V}_{s,m,\tilde\lambda,(1)}) \Norm{V}_{s,m,\tilde\lambda,(1)}.\]
The contribution of the third line is estimated in the same way. As for the contribution of the last line, we deduce from the explicit expression of $\Pi^\r$ that
\[\norm{\Pi^\r\partial^\k \partial_t^{j-1}G(V)}_{L^2}\leq \mu^{\frac12}\norm{\partial_t^{j-1}G(V)}_{H^{k+1}}\]
hence the contribution of the last line also satisfies the same estimates as above.
Finally, the contribution of the fourth line is estimated by
\begin{align*}\tilde\lambda^{\frac{m-j}2}\norm{[\partial^\k ,\S_t]\partial_t^j V}_{L^2}&\leq C(h_\star^{-1},\norm{V}_{H^s}) \tilde\lambda^{\frac{m-j}2}\norm{\partial_t^j V}_{H^{k-1}}\\
&\leq C(h_\star^{-1},\Norm{V}_{s,m,\tilde\lambda,(1)})\big(N^\r_{j,k-1,m,\tilde\lambda}+N^\s_{j,k-1,m,\tilde\lambda}\big).
\end{align*}
Altogether, by Cauchy-Schwarz inequality, we find for any  $m\le j\leq s$ and $k\in \NN$ such that $k\leq s-j$:
\begin{multline}\label{est-regular}
N^\r_{j,k,m,\tilde\lambda}\leq C(h_\star^{-1},\delta_\star^{-1},\Norm{V}_{s,m,\tilde\lambda,(1)})\big(N^\s_{j,k,m,\tilde\lambda}+ N^\r_{j,k-1,m,\tilde\lambda}+N^\s_{j,k-1,m,\tilde\lambda}+\Norm{V}_{s,m,\tilde\lambda,(1)}\big)\\
+\tilde\lambda^{-1/2}C(h_\star^{-1},\Norm{V}_{s,m,\tilde\lambda})\Norm{V}_{s,m,\tilde\lambda}.
\end{multline}

\subsubsection{Completion}

Assuming $\tilde\lambda\leq\lambda\mu$ and using that under the assumptions of Proposition~\ref{P.equiv}, one has
\[\big\vert N^\r_{j,0,m,\tilde\lambda}\big\vert^2+\big\vert N^\s_{j,0,m,\tilde\lambda}\big\vert^2=\tilde\lambda^{m-j}\norm{\partial_t^j V}_{L^2}^2\leq C(h_\star^{-1},\delta_\star^{-1},\norm{V}_{L^\infty})\Norm{V}_{s,m,\tilde\lambda,(1)}^2, \]
we immediately deduce from~\eqref{est-singular}-\eqref{est-regular} that
\begin{multline*}\Norm{V}_{s,m,\tilde\lambda,(2)}\leq  C(h_\star^{-1},\delta_\star^{-1},\Norm{V}_{s,m,\tilde\lambda,(1)})\Norm{V}_{s,m,\tilde\lambda,(1)} \\+\tilde\lambda^{-1/2}C(h_\star^{-1},\Norm{V}_{s,m,\tilde\lambda})\Norm{V}_{s,m,\tilde\lambda}.
\end{multline*}
It follows that there exists $\nu=C(h_\star^{-1},\Norm{V}_{s,m,\tilde\lambda})$ such that provided
\[\lambda\mu\geq \tilde\lambda \geq \nu,\]
the assumptions of Proposition~\ref{P.equiv} are satisfied with $\delta_\star=1/2$, and
\[\Norm{V}_{s,m,\tilde\lambda,(2)}\leq  C(h_\star^{-1},\Norm{V}_{s,m,\tilde\lambda,(1)})\Norm{V}_{s,m,\tilde\lambda,(1)}.\]
Proposition~\ref{P.bootstrap} is proved.

\section{Preparing the initial data}\label{S.Prep}

This section is dedicated to the proof of Theorem~\ref{T.Prep}. As in Section~\ref{S.WP}, we fix $\lambda,\mu\in  (0,+\infty)$ and assume for simplicity that 
\[\lambda\geq 1 \quad ; \quad \mu\leq 1 \quad ; \quad \lambda\mu\geq 1\]
the general setting being straightforwardly deduced.
We shall prove by induction on $m$ that we can set $c^{(j)}$ for $j\in\{1,\cdots,m\}$ such that~\eqref{est-prep-U} and~\eqref{est-prep-c} hold. We first notice that after differentiating~\eqref{FG-adim} with respect to time and using Lemma~\ref{L.products} (we constantly use this Lemma in the following when estimating nonlinear differential operators), one has that any solution $U=(\zeta,\u,\eta,w)$ to~\eqref{FG-adim} satisfies
\footnote{
Here and below, we denote
\begin{equation}\label{norm-prep}
\Norm{U}_{s,m}^2\eqdef\sum_{j=0}^m\norm{U}_{H^{s-j}}^2.
\end{equation}
}
\begin{equation}\label{dt-j+1 U}
\norm{\partial_t^{j+1}U}_{H^{s-(j+1)}}\leq C(h_\star^{-1},\Norm{U}_{j,s})\big(\Norm{U}_{s,j}+\lambda \Norm{\eta-h}_{s,j}\big).
\end{equation}
Hence we can focus on proving inductively that 
\[\lambda\Norm{ \eta^{(m)}-h^{(m)}}_{s,m}(0)\leq M_m\]
with $M_m=C(h_\star^{-1},M_0)M_0$. Notice that the result for $m=0$ is trivial and the result for $m=1$ follows 
from setting $c^{(1)}=-h_0\nabla\cdot\u_0$, as well as the identity
\begin{equation}\label{dt-(n-h)}
\partial_t(\eta-h)+\u\cdot\nabla(\eta-h)=w+h\nabla\cdot\u.
\end{equation}
Differentiating the above, and applying once again~\eqref{FG-adim} on the first-order time derivatives, we find that any solution to~\eqref{FG-adim} satisfies
\begin{equation}\label{dt2-(n-h)}
\partial_t^2(\eta-h)=\mathfrak{r}[U]+\lambda\mu\mathfrak{s}[U,\eta-h]-\lambda h^{-2} \mathfrak{t}[h](\eta-h)
\end{equation}
where $\mathfrak{r}$, $\mathfrak{s}$ and $\mathfrak{t}$ are nonlinear differential operators (in space) of order two. The large prefactor that $\lambda\mu$ in front of $\mathfrak{s}$ is compensated by the fact that this operator is quadratic in $\eta-h$, and hence we collect truly singular terms in the operator $\mathfrak{t}$:
\[\mathfrak{t}[h] \psi=\psi-\frac\mu3 h^3\nabla\cdot\left(h^{-1}\nabla \psi \right).\]
For future reference, we also notice that if $U=U^{(1)}_0\eqdef(\zeta_0,\u_0,h_0,-h_0\nabla\cdot\u_0)$, then $\mathfrak{s}[U_0^{(1)}]=0$ and
\begin{equation}\label{a1}
\mathfrak{r}[U_0^{(1)}]=h_0\left(\u_0\cdot \nabla (\nabla\cdot\u_0)-(\nabla\cdot\u_0)^2-\Delta \zeta_0-\nabla\cdot((\u_0\cdot\nabla)\u_0)\right).\end{equation}

Rooting from~\eqref{dt-(n-h)} and~\eqref{dt2-(n-h)}, we now define
\[\mathfrak{C}_{j}(U)\eqdef \partial_t^{j} \left(\mathfrak{r}[U]+\lambda\mu\mathfrak{s}[U,\eta-h]\right)-\lambda \big[\,\partial_t^{j}\,,\, h^{-2} \mathfrak{t}[h]\,\big](\eta-h)\]
and
\[\mathfrak{S}_m(U)\eqdef \sum_{k=0}^{\lfloor m/2\rfloor} (-\lambda h^{-2}\mathfrak{t}[h])^{k} \mathfrak{C}_{m-2k}(U)\]
so that any solution to~\eqref{FG-adim} satisfies for any $m\geq 2$
\begin{equation}\label{dtm-(n-h)}
\partial_t^m(\eta-h)-\mathfrak{S}_{m-2}(U)=
\begin{cases}
(-\lambda h^{-2}\mathfrak{t}[h] )^{m/2}(\eta-h)& \text{if $m$ is even,}\\
(-\lambda h^{-2}\mathfrak{t}[h])^{(m-1)/2}\partial_t(\eta-h)& \text{if $m$ is odd.}
\end{cases}
\end{equation}
We deduce the following expression for $c^{(m)}$:
\begin{equation}\label{def-cj}
(- h_0^{-2}\mathfrak{t}[h_0])^{\lfloor m/2\rfloor } c^{(m)} =-\mathfrak{S}_{m-2}(U_0^{(m-1)})- (-\lambda h_0^{-2}\mathfrak{t}[h_0])^{\lfloor m/2\rfloor }s^{(m-2)}.
\end{equation}
where $\mathfrak{S}_{m-2}(U_0^{(m-1)})$ is the differential operator of order $m$ obtained when all time derivatives have been replaced by spatial derivatives through~\eqref{FG-adim}, and
\[s^{(m)}\eqdef \begin{cases}
\sum_{k=1}^{m/2} \lambda^{-k} c^{(2k)} &\text{ if $m$ is even,}\\
\sum_{k=1}^{(m-1)/2} \lambda^{-k} c^{(2k+1)}-\sum_{k=1}^{(m+1)/2} \lambda^{-k} \u_0\cdot\nabla c^{(2k)}&\text{ if $m$ is odd.}
\end{cases}
\]
Notice $c^{(m)}$ is well-defined by~\eqref{def-cj} and induction on $m$, using Lemma~\ref{L.c}. We prove below that this choice allows to obtain the desired estimates.

Assuming $m$ is even for simplicity (the case $m$ odd is treated in the same way, with straightforward adjustments), we have by~\eqref{def-cj}
\[
 (-\lambda h_0^{-2}\mathfrak{t}[h_0])^{m/2}s^{(m-2)}=- \lambda h_0^{-2}\mathfrak{t}[h_0]\mathfrak{S}_{m-4}(U_0^{(m-3)}).\]
Now, we have by repeated use of~\eqref{FG-adim} 
and direct product estimates that
\begin{multline*}\norm{\lambda h_0^{-2}\mathfrak{t}[h_0] \big(\mathfrak{S}_{m-4}(U_0^{(m-3)})-\mathfrak{S}_{m-4}(U_0^{(m-1)})\big)}_{H^{s-m}}\\
\leq C_{m}' \lambda^{1+\frac{m-4}2} \left( \norm{U_0^{(m-1)}-U_0^{(m-3)}}_{H^{s-m}} +\mu^{\frac{m-2}2}\norm{U_0^{(m-1)}-U_0^{(m-3)}}_{H^{s-2}} \right)
\end{multline*}
where $C_{m}'=C(h_\star^{-1},\norm{U_0^{(m-1)}}_{H^s},\norm{U_0^{(m-3)}}_{H^s})$. Moreover, we have by definition
 \[- \lambda h_0^{-2}\mathfrak{t}[h_0]\mathfrak{S}_{m-4}(U_0^{(m-1)})
=\mathfrak{S}_{m-2}(U_0^{(m-1)})-\mathfrak{C}_{m-2}(U_0^{(m-1)}) \]
and
\[\norm{\mathfrak{C}_{m-2}(U_0^{(m-1)})}_{H^{s-m}}\leq C_m
\Big(\Norm{U^{(m-1)}}_{s,m-2}+\lambda \Norm{\eta^{(m-1)}-h^{(m-1)}}_{s,m-3} \Big)
\]
with $C_m=C(h_\star^{-1},\Norm{U^{(m-1)}}_{s,m-2},\lambda \Norm{\eta^{(m-1)}-h^{(m-1)}}_{s,m-3})$.
Combining the above and using the induction hypotheses~\eqref{est-prep-U} and~\eqref{est-prep-c}, we find
\[\norm{\mathfrak{S}_{m-2}(U_0^{(m-1)})+ (-\lambda h_0^{-2}\mathfrak{t}[h_0])^{m/2}s^{(m-2)}}_{H^{s-m}}\leq  C(h_\star^{-1},M_0)M_0.\]
It follows that by Lemma~\ref{L.c} that $c^{(m)}$ is well-defined by~\eqref{def-cj} and satisfies 
\[\norm{c^{(m)}}_{H^{s-m}}+\mu^{m/2}\norm{c^{(m)}}_{H^{s}}\leq  C(h_\star^{-1},M_0)M_0.\]
Notice that we have in particular, since $\lambda\mu\geq 1$, $\norm{U_0^{(m)}}_{H^s}\leq C(h_\star^{-1},M_0)M_0$.
We also observe that for any $j\leq m$, one has as above
\[
\lambda\norm{\mathfrak{S}_{j-2}(U_0^{(m-1)})-\mathfrak{S}_{j-2}(U_0^{(m)})}_{H^{s-j}}\leq C(h_\star^{-1},M_0)M_0 .\]
Using this estimate with $j=m$ in~\eqref{def-cj}, plugging into~\eqref{dtm-(n-h)} and using the definition~\eqref{def-id} shows that with our choice of $c^{(j)}$, one has 
\[\lambda \norm{\partial_t^m(\eta^{(m)}-h^{(m)})}_{H^{s-m}} (0)\leq  C(h_\star^{-1},M_0)M_0.\]
The corresponding estimates for time derivatives of lower order are obtained using the estimate directly into~\eqref{dtm-(n-h)} and using the induction hypothesis. 
Hence we proved
\[\lambda\Norm{ \eta^{(m)}-h^{(m)}}_{s,m}(0)\leq  C(h_\star^{-1},M_0)M_0,\]
and we deduce from~\eqref{dt-j+1 U}
\[\Norm{ U^{(m)}}_{s,m+1}(0)\leq  C(h_\star^{-1},M_0)M_0.\]

This completes the inductive proof of~\eqref{est-prep-U} and~\eqref{est-prep-c}. We have already displayed $c^{(1)}=-h_0\nabla\cdot\u_0 $, and~\eqref{a1} with~\eqref{def-cj} yields 
\begin{align*}( h_0^{-2}\mathfrak{t}[h_0])c^{(2)} &=\mathfrak{C}_0(U_0^{(m-1)})\\
&=h_0\left(\u_0\cdot \nabla (\nabla\cdot\u_0)-(\nabla\cdot\u_0)^2-\Delta \zeta_0-\nabla\cdot((\u_0\cdot\nabla)\u_0)\right),
\end{align*}
from which we deduce the explicit expression for $c^{(2)}$.

\section{Conclusion}\label{S.conclusion}

We have shown the relevance of the Favrie-Gavrilyuk system for producing approximate solutions to the Green-Naghdi system ---and ultimately the water-waves system. To this aim, we have exhibited the role of the shallowness parameter, which may induce undesirable oscillations in space in the shallow-water regime. In order to avoid these oscillations it appears necessary ---or at least advisable--- to suitably set the initial data for the augmented variables $\eta,w$. The following setting is expected to produce good results: after non-dimensionalizing the equations, set $\lambda\gtrsim \mu^{-1}$ where $\mu$ is the shallowness parameter, and given the initial data $h\id{t=0}=h_0$ and $\u\id{t=0}=\u_0$, let

\[ w\id{t=0}=-h_0\nabla\cdot\u_0  \quad \text{ and }\quad  \eta\id{t=0}=h_0+\lambda^{-1} c \] 
where $c$ is the unique solution to
\[h_0^{-3}\mathfrak{t}[h_0] c=-\Delta \zeta_0+\u_0\cdot \nabla (\nabla\cdot\u_0)-(\nabla\cdot\u_0)^2-\nabla\cdot((\u_0\cdot\nabla)\u_0)\]
with
\[\mathfrak{t}[h_0] c\eqdef c-\frac\mu3 h_0^3\nabla\cdot\left(h_0^{-1}\nabla c \right).\]

One of the main challenges for future studies on the Favrie-Gavrilyuk system would consist in taking into account variations of the bottom topography, which yield new singular terms, but with variable coefficients. We refer to~\cite{BreschMetivier10} for a related problem. Proposing a well-adapted  numerical scheme will also most certainly require a tailored analysis, in particular due to the fact that the linearized system is not uniformly stable as $\lambda\to \infty$.

\section*{Acknowledgments}

The author is grateful to Nicolas Favrie and Sergey Gavrilyuk for enlightening discussions and encouragement.

\end{document}